\documentclass{article}
\usepackage{graphicx}
\usepackage{amsfonts}
\usepackage{amsthm}
\usepackage{mathtools}
\usepackage{mathrsfs}
\usepackage{textcomp}
\usepackage{caption}
\usepackage{setspace}
\usepackage{bm}
\usepackage{authblk}
\newtheorem{theorem}{Theorem}
\newtheorem{lemma}{Lemma}

\newtheorem{definition}{Definition}
\newtheorem*{remark}{Remark}

\newcommand\extrafootertext[1]{%
    \bgroup
    \renewcommand\thefootnote{\fnsymbol{footnote}}%
    \renewcommand\thempfootnote{\fnsymbol{mpfootnote}}%
    \footnotetext[0]{#1}%
    \egroup
}

\begin{document}
\title{\bf The Change of Basis Groupoid}
\markright{The Change of Basis Groupoid}
\author{D.A.\ Wolfram}
\affil{\small College of Engineering \& Computer Science\\The Australian National University, Canberra, ACT 0200\\
\medskip
{\rm David.Wolfram@anu.edu.au}
}
\date{}

\maketitle
\maketitle

\begin{abstract}
We show that the change of basis matrices of a set of $m$ bases of a finite vector space is a connected groupoid of order $m^2$. 
We define a general method to express the elements of change of basis matrices as algebraic expressions using optimizations of evaluations of vector dot products.  Examples are given with orthogonal polynomials.
\end{abstract}

\extrafootertext{MSC: Primary 15A03; Secondary 20N02, 33C45}




\section{Background and Related Work}

Change of basis in a finite vector space has numerous significant and widespread applications in scientific computing and engineering.
Two application areas are image, video and data compression~\cite{strang93}, e.g.  it can be used to implement DCT-II in JPEG compression~\cite{pennebaker}; and improving properties of spectral methods for solving differential equations numerically. These properties include, better convergence~\cite{hale1}, lower computational complexity~\cite{fortunato,olver}, and better numerical stability~\cite{olver}. 


Much work has been done on numerical solutions for finding the coefficients in Legendre expansions, e.g.\ see the summary in Hale~\cite{hale}. From equation (1.2) of Hale and Townsend~\cite{hale}, these coefficients are $c_n^{leg}$ where
\[
p_N(x) = \sum_{n=0}^N c_n^{leg} P_n(x)
\] and $x \in [-1,1]$. The terms $P_n(x)$ on the right side of the equation are Legendre polynomials. The coefficients $c_n^{leg}$ are represented by floating-point numbers. The function $p_N(x)$ can also be expressed using Chebyshev polynomials of the first kind
\[
p_N(x) = \sum_{n=0}^N c_n^{cheb} T_n(x)
\] and $x \in [-1,1]$. They  provide an algorithm for implementing the transform between the coefficients $c_n^{leg}$ and $c_n^{cheb}$ with an accuracy of fifteen decimal places~\cite{hale}.

These transforms are scaled changes of bases between Legendre and Chebyshev polynomials of the first kind. For example, let $M_{TP}$ be the change of basis matrix from $P_n(x)$ to $T_n(x)$ where $0 \leq n \leq 4$,

\[
M_{TP} = 
\begin{bmatrix}
 1 & 0 & \frac14 & 0 & \frac9{64} \\
 \\
 0 & 1 & 0 & \frac38 & 0 \\
 \\
 0 & 0 & \frac34 & 0 & \frac5{16} \\
 \\
 0 & 0 & 0 & \frac58 & 0 \\
 \\
 0 & 0 & 0 & 0 & \frac{35}{64} \\
\end{bmatrix}
\mbox{ so that }
M_{TP}
\begin{bmatrix}
 c_0^{leg}\\
 \\
 c_1^{leg}\\
 \\
 c_2^{leg}\\
 \\
 c_3^{leg}\\
 \\
 c_4^{leg}
\end{bmatrix} 
= 
\begin{bmatrix}
 c_0^{cheb}\\
 \\
 c_1^{cheb}\\
 \\
 c_2^{cheb}\\
 \\
 c_3^{cheb}\\
 \\
 c_4^{cheb}
\end{bmatrix}
\]

In this paper, we are concerned with defining algebraic expressions for the elements in change of basis matrices, rather than numerical numerical methods for approximating them.

For example, from  Szeg\"{o}~\cite{szego} we can show that for the change of basis matrix from $P_n(x)$ to $T_n(x)$, we have
\[
P_n(x) = \sum_{k=0}^{\lfloor \frac{n}2 \rfloor} \beta(n, k) T_{n - 2k}(x)
\] where 
\begin{equation}
\beta(n, k) =  \left\{
\begin{array}{ll}  
2^{1 - 2n}{{2k} \choose k}{{2n - 2k} \choose {n-k}} & \mbox{if $2k < n$}\\
\\
2^{- 2n}{{n} \choose {\frac{n}{2}}}^2 & \mbox{if $2k = n$}\\
\end{array}
\right.
\end{equation} and $n \geq 0$. The function $\beta$  is defined by an algebraic expression and it evaluates to the elements  in the upper triangle of $M_{TP}$, e.g. $\beta(3, 1) = \frac38$ and $\beta(4, 2) = \frac{9}{64}$. There are other equivalent expressions for $\beta$, e.g.~\cite[equation 22.3.13]{as}.

Apart from some specific cases such as this, there do not seem to be similar algebraic expressions for transforms between other bases of orthogonal polynomials, such as between Hermite polynomials and Chebyshev polynomials.

This leads to the question of finding the structure for changes of bases of a finite-dimensional vector space. We show that this structure is a connected groupoid~\cite{brandt,brown} that we call the ``change of basis groupoid''.

We focus on finding the elements of the change of basis matrices  for  transforms between any pair of bases for the same vector space as evaluations of a function expressed algebraically. We call these functions ``coefficient functions''. The function  $\beta$ above is an example of this.

We then consider optimizations  for finding coefficient functions based on excluding elements that are known to be zero. 
The first of these involves excluding elements in the lower triangles of matrices. This occurs when the basis polynomials have definite parity and we are only interested in even or odd functions. This is related to the  Schuster-Dilts triangular matrix acceleration optimization that is used in representing a function with spherical harmonics~\cite[\S18.12.6]{boyd}.  Other optimizations are excluding specific terms in basis polynomials, and excluding terms when one basis has vectors with definite parity and the other does not. An example of this is a change of basis between Hermite and  Laguerre polynomials.

Related research has concerned defining recurrence relations for finding the coefficients of  series expansions expressed with  orthogonal polynomials~\cite{lewanowicz,salzer}.


%

\section{Change of Basis is Inherent in Defining a Basis}

We consider vector spaces with a basis of the form $B = \{s_0, s_1, \ldots \}$ where each vector $v$ in the vector space can be defined uniquely as a finite linear combination of the basis vectors, i.e.
\begin{equation} \label{lincomb}
v = \sum_{k= 0}^n a_k s_k
\end{equation} where the $a_k \in \mathbb{R}$.

The coordinates of vector $v$ with respect to $B$ are another way to represent this.
\begin{definition}
Let $V$ be a vector space of finite dimension with coefficients over a given field $F$, and $B = \{v_0, \ldots, v_n\}$ be a basis of $V$.
For all vectors $v \in V$ such that $v = \sum_{k= 0}^n a_k v_k$ where $a_k \in F$ and $0 \leq k \leq n$, the coordinate vector of $v$ is
\[
(a_0, a_1, \ldots, a_n).
\] 
\end{definition}
Generally, the set of  basis vectors  $\{s_0, s_1, \ldots \}$  is defined using another known basis. When the vector space is $\mathbb{R}^n$, this is usually the standard basis. When the vector space is the ring of polynomials $\mathbb{R}[X]$, the basis is usually the monomials $\{1, x, x^2, \ldots \}$. 

\begin{remark}
We assume that a basis $\{v_0, v_1, \ldots, v_{n-1}\}$ of a vector space is an ordered basis, so that a coordinate vector of the form $(c_0, c_1, \ldots, c_{n-1})$ with respect to this basis represents the vector $c_0 v_0  + c_1 v_1 + \cdots c_{n-1} v_{n-1}$.
\end{remark}

Defining a vector basis inherently involves defining a change of basis from a known basis. For example, consider the vector basis of the Chebyshev polynomials of the first kind:
\begin{align*}
T_0(x) =& 1\\
T_1(x) =& x\\
T_2(x) =& 2x^2 -1\\
T_3(x) =& 4x^3 - 3x
\end{align*}

These basis vectors are defined in terms of the monomials. In matrix form we have a change of basis matrix from Chebyshev polynomials $T$ up to $T_3(x)$ to the monomials $M$ where the transposes of the coefficient vectors  of the domain basis vectors form its columns.
\begin{gather*}
\begin{bmatrix}
1 & 0 & -1 & \phantom{-}0\\
0 & 1 & \phantom{-}0 & -3\\
0 & 0 & \phantom{-}2 & \phantom{-}0\\
0& 0 & \phantom{-}0& \phantom{-}4
\end{bmatrix}
\end{gather*}

This defines a mapping from the coordinates vectors defined using the basis $\{T_0(x), \ldots, T_3(x)\}$ to the coordinate vectors of the same vector using the basis $\{1, x, x^2, x^3\}$. For example
\begin{gather*}
\begin{bmatrix}
1 & 0 & -1 & \phantom{-}0\\
0 & 1 & \phantom{-}0 & -3\\
0 & 0 & \phantom{-}2 & \phantom{-}0\\
0& 0 & \phantom{-}0& \phantom{-}4
\end{bmatrix}
\begin{bmatrix}
\phantom{-}2\\
\phantom{-}0\\
-1\\
\phantom{-}1
\end{bmatrix} = 
\begin{bmatrix}
\phantom{-}3\\
-3\\
-2\\
\phantom{-}4
\end{bmatrix} 
\end{gather*} and

\[
T_3(x) - T_2(x) + 2 T_0(x)  = 4x^3 - 2x^2 - 3x + 3 
\] so that $(2, 0, -1, 1)$ is mapped to $(3, -3, -2, 4)$.

\subsection{Change of Basis Matrix}
In the following, we consider triangular bases of a vector space $V$.
Defining a  basis $\{s_0, s_1, \ldots, s_{n-1} \}$ of $V$  involves defining a mapping of it to another  basis $\{t_0, t_1, \ldots , t_{n-1} \}$ of $V$.
The standard basis for defining a basis in  $\mathbb{R}^n$ , or using the monomial basis for defining a basis in $\mathbb{R}[X]$ are just two possibilities. 

Instead of the monomials,  we can use the shifted Legendre polynomials for example to represent the basis of Chebyshev polynomials of the first kind. The first four shifted Legendre polynomials represented using monomials are
\begin{align*}
\tilde{P}_0(x) =& 1\\
\tilde{P}_1(x) =& 2x -1\\
\tilde{P}_2(x) =& 6x^2 - 6x +1\\
\tilde{P}_3(x) =& 20x^3 - 30x^2 + 12x -1.
\end{align*}

We express $T_3(x)$ uniquely in terms of the shifted Legendre polynomials by 
\[
T_3(x) = \frac{1}{5} \tilde{P}_3(x) + \tilde{P}_2(x) +\frac{3}{10}\tilde{P}_1(x) - \frac{1}{2} \tilde{P}_0(x)
\] and similarly for all other basis vectors in $\{T_0(x), T_1(x), \ldots \}$.

A change of basis matrix 
maps the coordinate vector of a vector $v$ defined using the basis of Chebyshev polynomials $T$  to its coordinate vector when $v$ is defined using the basis of shifted Legendre polynomials. For example, when $n = 4$, this matrix is 
\begin{gather*}
\begin{bmatrix}
1 & \frac12 & -\frac13 & -\frac12\\
\\
0 & \frac12 & \phantom{-}1 & \phantom{-}\frac3{10}\\
\\
0 & 0 & \phantom{-}\frac13 & \phantom{-}1\\
\\
0& 0 & \phantom{-}0& \phantom{-}\frac15
\end{bmatrix}
\end{gather*}

\section{Change of Basis Groupoid}

Let $V$ be a vector space that has bases $s = \{s_0, s_1, \ldots, s_{n}\}$ and $t = \{t_0, t_1, \ldots , t_{n}\}$.

The following equation expresses the mapping of a change of basis matrix $M_{ts}$.
\begin{equation}
M_{ts} b_s = a_t 
\end{equation}
The subscripts are simple types~\cite{church}. The columns of $M_{ts}$ are the transposed coordinate vectors with respect to $t$ of the basis vectors of $s$. 
%


\begin{lemma} \label{inv-cob}
For every change of basis matrix $M_{ts}$, its inverse $M_{ts}^{-1}$ exists and $M_{ts}^{-1} = M_{st}$.
\end{lemma}

\begin{proof}

$M_{ts}$ is a square matrix whose columns are the transposed coordinate vectors with respect to $t$ of the basis vectors of $s$ . These columns are linearly independent, so that $M_{ts}$ is invertible by the Fundamental Theorem of Invertible Matrices, e.g \cite[Theorem 3.27]{poole}.


To show $M_{ts}^{-1} = M_{st}$, consider
any column matrices $a$ and $b$
such that $M_{ts} b = a$. 

Now $M_{ts}^{-1}M_{st} b = M_{ts}^{-1}a$. It follows that $M_{ts}^{-1} a = b$ by the associativity of matrix multiplication. If $a$ is known   then  $M_{ts}^{-1}$ is the matrix that can be similarly used to find $b$, as required.
\end{proof}


\begin{definition}
Two change of basis matrices $M_{ts}$ and $M_{vu}$ are equal if and only if they have  the same elements, $t=v$ and $s=u$.
\end{definition}

\begin{lemma} \label{basis-unique}
Every  change of basis matrix is unique.
\end{lemma}

\begin{proof}

Suppose $M= M_{ts}$ is a change of basis matrix.  If $M$ is not unique, there is a change of basis matrix $N = N_{ts}$ that has the same dimensions as $M$ and $M \not= N$. 
The matrix $M^{-1}$ exists from Lemma~\ref{inv-cob}, and $N$ is its right inverse. We have $M = M (M^{-1} N)$  and $ M (M^{-1} N) = (M M^{-1}) N$, so that $M = N$ which is a contradiction that they are unequal. \end{proof}

%
%

\begin{theorem} \label{itsagroupoid}
Let $V$ be a vector space and suppose there is a set $S$  of $m$ bases of $V$. 
The set of change of basis matrices $M_{ts}$ where $s, t \in S$ is a connected groupoid $G$ of order $m^2$.
\end{theorem}

\begin{proof}
The axioms of the algebraic definition of a groupoid~\cite{brandt,brown} are met, i.e. 
\begin{itemize}
\item The elements of $G$ are the change of basis matrices and the identity matrices $M_{tt}$ where $t \in S$ and they all have the same dimensions.
\item There is a partially defined associative binary operation $\circ$  on the  elements of $G$. For every $M, N \in G$, the operation $\circ$  is defined if and only if $M$ has the form $M_{ut}$, and $N$ has the form $N_{ts}$ in which case $M \circ N = M_{us} \in G$. The operation $\circ$ is matrix multiplication and we omit the symbol.
\item If $M = M_{ts}$ is any element of the groupoid, then its inverse $M^{-1} = M_{st}$ is an element of the groupoid and it is defined by matrix inversion of $M$ following Lemma~\ref{inv-cob}.
\item If $M = M_{ts}$ is any element of the groupoid, then $M M^{-1} = M_{tt}$ is an identity element of the groupoid.
\end{itemize}

The groupoid is connected or transitive, because for any two bases $s$ and $t$  in $S$, there exists a matrix $M_{ts} \in G$.

For each of the $m$ bases there are $m^2$ permutations of them taken two at a time with repetitions, which is the order of the groupoid following Lemma~\ref{basis-unique}.  
\end{proof}

\subsection{Example}
The set $S$ comprises the three bases: monomials $M$; Legendre polynomials $P$ expressed in terns of monomials; and Chebyshev polynomials of the first kind $T$ expressed in terns of monomials. All basis vectors have even degree from $0$ to $6$.
The change of basis matrix 
\[
M_{TP} = 
\begin{bmatrix}
1 & \frac{1}{4} & \frac{9}{64} & \frac{25}{256}\\
\\
0 & \frac{3}{4} & \frac{5}{16} & \frac{105}{512}\\
\\
0 & 0 & \frac{35}{64} & \frac{63}{256}\\
\\
0 & 0 & 0 &\frac{231}{512}
\end{bmatrix}
\]

This gives
\[
\begin{bmatrix}
1 & \frac{1}{4} & \frac{9}{64} & \frac{25}{256}\\
\\
0 & \frac{3}{4} & \frac{5}{16} & \frac{105}{512}\\
\\
0 & 0 & \frac{35}{64} & \frac{63}{256}\\
\\
0 & 0 & 0 &\frac{231}{512}
\end{bmatrix}
\begin{bmatrix}
0\\
\\
1 \\
\\
0\\
\\
2
\end{bmatrix}
=
\begin{bmatrix}
\frac{57}{128}\\
\\
\frac{297}{256} \\
\\
\frac{63}{128}\\
\\
\frac{231}{256}
\end{bmatrix}
\]

In terms of polynomials, we have
\begin{align*}
2 P_6(x) + P_2(x) &= \frac{231}{256} T_6(x) + \frac{63}{128} T_4(x) + \frac{297}{256} T_2(x) + \frac{57}{128}T_0(x)\\
&= \frac{1}8(231 x^6 - 315 x^4 + 117 x^2 - 9).
\end{align*}

From Boyd~\cite{boyd}, the matrix 
\[
M_{PT} = 
\begin{bmatrix}
1 & -\frac{1}{3} & -\frac{1}{15} & -\frac{1}{35}\\
\\
0 & \frac{4}{3} & -\frac{16}{21} & -\frac{4}{21}\\
\\
0 & 0 & \frac{64}{35} & -\frac{384}{385}\\
\\
0 & 0 & 0 &\frac{512}{231}
\end{bmatrix}
\]
which is the inverse of $M_{TP}$. The matrix $M_{MP}$ is
\[
\begin{bmatrix}
1 & -\frac{1}{2} & \frac{3}{8} & -\frac{5}{16}\\
\\
0 & \frac{3}{2} & -\frac{15}{4} & \frac{105}{16}\\
\\
0 & 0 & \frac{35}{8} & -\frac{315}{16}\\
\\
0 & 0 & 0 &\frac{231}{16}
\end{bmatrix}
\]

From these we can find the other groupoid elements, $ M_{PM}$, $M_{TM} = M_{TP} M_{PM}$, $M_{MT}$,  $I_{MM}$, $I_{PP}$ and $I_{TT}$. There are $m^2 =9$ elements in total. There are three identity matrices, but they are unequal because of their types.

\subsection{The Exchange Basis}
\label{EB}
\begin{definition} \label{exbasis}
Given two change of basis mappings ${\cal T}_2: c \rightarrow b$  and  ${\cal T}_1: b \rightarrow d$,   the  exchange basis is $b$.
\end{definition}

Suppose that we wish to find $\alpha_3$ for a pair of bases  $s$ and $t$ of a subspace of $\mathbb{R}[X]$.
It is convenient to use the monomials $M$ as the exchange basis because $\alpha$ is sometimes known in the literature~\cite{as,dlmf,handscomb} for   ${\cal T}_2: s \rightarrow M$ and ${\cal T}_1: M \rightarrow t$.

\section{Optimizing with the Triangular Basis}

From the change of basis groupoid, given two change of basis matrices $M_{cb} M_{bd}$ the elements of $M_{cd}$ are formed from the vector dot products of the rows of $M_{cb}$ and the columns of $M_{bd}$. Bases that satisfy the following condition lead to an optimization of these vector dot products.
\begin{definition} \label{triangular-basis}
Let $V$ be a vector space of finite dimension $n > 0$ and $B_1$ and $B_2$ be bases of $V$.

The basis $B_2$ is a triangular basis with respect to $B_1$ if and only if there is a permutation $\{v_1, \ldots , v_n \}$ of the coordinate vectors of $B_2$ with respect to $B_1$ such that the  $n \times n$ matrix $\left[ v_1^T, \ldots , v_n^T \right]$ is an upper triangular matrix.
\end{definition}

%
%
%
%

An example of an orthogonal basis in $\mathbb{R}^8$ that is not a triangular basis is the Haar function wavelet matrix~\cite{strang93}:
\begin{gather*}
W_8 = 
\begin{bmatrix}
1 & \phantom{-}1 & \phantom{-}1 & \phantom{-}0& \phantom{-}1&\phantom{-}0 &\phantom{-}0&\phantom{-}0\\
1 & \phantom{-}1 & \phantom{-}1 &\phantom{-} 0 & -1& \phantom{-}0&\phantom{-}0&\phantom{-}0\\
1 & \phantom{-}1 & -1 & \phantom{-}0 &\phantom{-}0 &\phantom{-}1&\phantom{-}0&\phantom{-}0\\
1& \phantom{-}1 & -1&  \phantom{-}0 &\phantom{-}0 & -1&\phantom{-}0&\phantom{-}0\\
1 & -1 & \phantom{-}0 & \phantom{-}1 &\phantom{-}0&\phantom{-}0& \phantom{-}1&\phantom{-}0\\
1 & -1 & \phantom{-}0 &\phantom{-}1 &\phantom{-}0 &\phantom{-}0& -1&\phantom{-}0\\
1 & -1 & \phantom{-}0& -1 & \phantom{-}0&\phantom{-}0&\phantom{-}0&\phantom{-}1\\
1 &-1 & \phantom{-}0& -1 &\phantom{-}0&\phantom{-}0&\phantom{-}0& -1
\end{bmatrix}
\end{gather*}

The basis vectors $(1, 1, 1, 1, 1, 1, 1, 1)$ and $(1, 1, 1, 1, -1, -1, -1, -1)$  with respect to the standard basis both have non-zero coordinates in the eight dimension.

An example of a triangular basis is Chebyshev polynomials of the first kind of odd degree up to degree $5$. This basis is
\[
\{ x, \ \ 4x^3 - 3x, \ \  16 x^5 -20 x^3 + 5x\}
\]
The basis vectors of this subspace of $\mathbb{R}[X]$  have the following coordinate vectors wth respect to the basis $\{x, x^3, x^5\}$:
\[
(1, 0, 0), \ \ (-3, 4, 0), \ \ (5, -20, 16)
\] and the matrix
\begin{gather*}
\begin{bmatrix}
1 & -3 & \phantom{-}5 \\
0 & \phantom{-}4 & -20\\
0 & \phantom{-}0 & \phantom{-}16
\end{bmatrix}
\end{gather*} is an upper triangular matrix.

This optimization is related to the Schuster-Dilts triangular matrix acceleration optimization. It was found independently by Schuster~\cite{schuster} and Dilts~\cite{dilts}. Boyd~\cite[\S18.12.6]{boyd} gives an example of it involving a change of basis from $T_n(x)$ to $P_n(x)$ where they have even parity.

By excluding elements below the main diagonal, a $m \times m$ matrix can be represented with $\frac12 (m^2 + m)$ elements.

%
%
%

\section{The Dot Product Solution Method}

We consider a method of defining algebraically a change of basis function between two triangular bases of a vector space. The method uses the matrix  dot product of change of basis matrices.

\begin{definition}\label{coefffunction}
Suppose that $s$ and $t$ are triangular bases of a vector space $V$.
The mapping $\mathcal{T}: s \rightarrow t$ satisfies

\begin{equation} \label{mapF}
s_n= \sum_{k=0}^{n} \alpha(n, k)  t_{n - k}
\end{equation}where $\alpha(n, k) \in \mathbb{R}$, i.e. each basis vector of $s$ is a unique linear combination of the basis vectors of $t$.
The function $\alpha$ is called a coefficient function.
\end{definition}

There is a change of basis matrix equation from equation (\ref{mapF}):

\begin{gather*}
\begin{bmatrix}
\vdots & \cdots & \vdots & \vdots & \vdots \\
0 & \cdots & \alpha(n-3, 0) & \alpha(n-2, 1) & \alpha(n-1, 2)& \alpha(n ,3)\\
0 &\cdots & 0 & \alpha(n-2, 0) & \alpha(n-1, 1)& \alpha(n ,2)\\
0 & \cdots &0 & 0 & \alpha(n-1, 0)& \alpha(n ,1)\\
0 & \cdots &0 & 0 & 0 & \alpha(n, 0)
\end{bmatrix}
\begin{bmatrix}
\vdots \\
b_3 \\
b_2\\
b_1 \\
b_0
\end{bmatrix}
= 
\begin{bmatrix}
\vdots \\
a_3 \\
a_2\\
a_1\\
a_0
\end{bmatrix}
\end{gather*}

The matrix on the left is a square matrix. This represents a function that maps a a matrix of coordinates of a vector in the domain basis $s$ in the second matrix to the last matrix which is the coordinates of the same vector in the range basis $t$.
Given a change of basis mapping that satisfies equation~(\ref{mapF}), 
a solution method is a method to find algebraic expressions for the coefficient function $\alpha(n, k)$.

For example, when $n=6$, $M_{ts}$ is a $7 \times 7$ matrix. When it is applied to the coordinate matrix for $s_3(x)$ 
\begin{gather*}
M_{ts}
\begin{bmatrix}
0\\
0\\
0\\
1\\
0\\
0\\
0\\
\end{bmatrix}
= 
\begin{bmatrix}
\alpha(3,3)\\
\alpha(3,2)\\
\alpha(3,1)\\
\alpha(3,0)\\
0\\
0\\
0\\
\end{bmatrix}
\end{gather*}
it produces the fourth column of $M_{ts}$ that has the following coordinates of $s_3(x)$ in the basis $t$:
$(\alpha(3,3), \alpha(3,2), \alpha(3,1), \alpha(3,0), 0, 0, 0)$. From equation~\ref{mapF}, we have
\[
s_3(x) = \sum_{k=0}^{3} \alpha(3, k)  t_{3 - k}(x).
\]

Given two  change of basis matrices of the same dimensions, $M_{ts}$ and $M_{sv}$  we have $M_{tv} = M_{ts} M_{sv}$ from Lemma~\ref{itsagroupoid}. 

More specifically, if the elements in $M_{ts}$ are applications of $\alpha_1$ and those of $M_{sv}$ are $\alpha_2$, then we can find the elements that are applications of $\alpha_3$ in $M_{tv}$ by matrix multiplication

The leftmost matrix is a square upper triangular matrix with $n+1$ rows and columns. 
The expression for $\alpha_3$ is found from the dot product of the $i$th row of $M_{ts}$ with the $j$th column of $M_{sv}$. Given such a $h \times h$ matrix $M$, we can show that the entry $M_{i, j} = \alpha(n - (h-j), j -i)$ where $1 \leq i, j \leq h$ and $j \geq i$. We use the convention that the sum below equals $0$ when $i > j$.

 
\begin{equation} \label{alpha-new}
\alpha_3(n - (h-j), j-i)= \sum_{v=i}^{j} \alpha_1(n - (h-v), v -i) \alpha_2(n - (h-j), j - v)
\end{equation}
The entry for $\alpha(m, k)$ is in row $m-n + h - k$ and column $m-n + h$ where $0 \leq k \leq m$ and $n - h + 1 \leq  m \leq n$.

It is convenient to consider the dot product with the last column of $M_{sv}$ when $j = h$. 
The matrix equation then has the following form.
\begin{gather*}
\begin{bmatrix}
\vdots & \vdots & \vdots & \vdots & \vdots \\
0 & \alpha_1(n-3, 0) & \alpha_1(n-2, 1) & \alpha_1(n-1, 2)& \alpha_1(n ,3)\\
0 & 0 & \alpha_1(n-2, 0) & \alpha_1(n-1, 1)& \alpha_1(n ,2)\\
0 & 0 & 0 & \alpha_1(n-1, 0)& \alpha_1(n ,1)\\
0 & 0 & 0 & 0 & \alpha_1(n, 0)
\end{bmatrix}
\begin{bmatrix}
\vdots \\
\alpha_2(n, 3) \\
\alpha_2(n, 2)\\
\alpha_2(n, 1) \\
\alpha_2(n, 0)
\end{bmatrix}
= 
\begin{bmatrix}
\vdots \\
\alpha_3(n, 3) \\
\alpha_3(n, 2)\\
\alpha_3(n, 1)\\
\alpha_3(n, 0)
\end{bmatrix}
\end{gather*}

Equation~(\ref{alpha-new}) simplifies to
\begin{equation} \label{alpha-new2}
\alpha_3(n, k)= \sum_{v=0}^{k} \alpha_1(n - v, k-v) \alpha_2(n , v)
\end{equation} where $0 \leq k \leq n$, which provides a solution.

\subsection{Further Optimizations}
\label{opt-subspace}

The evaluation of $\alpha_3(n, k)$ in equation (\ref{alpha-new2}) can be optimized when there are values of $\alpha_1(n -v, k-v)$ or $\alpha_2(n, v)$ that can be excluded. For example, Chebyshev polynomials of the first  kind $T_n(x)$ have definite parity~\cite[Definition 21, \S 8.2]{boyd}. The change of basis matrix of them up to order $4$ with respect to the monomials is
\begin{gather*}
M_{MT} = 
\begin{bmatrix}
1 & 0 & -1& \phantom{-}0& \phantom{-}1\\
0 & 1 & \phantom{-}0 & -3& \phantom{-}0\\
0 & 0 & \phantom{-}2 &\phantom{-} 0& -8\\
0 & 0 & \phantom{-}0 & \phantom{-}4 & \phantom{-}0\\
0&0&\phantom{-}0&\phantom{-}0&\phantom{-}8
\end{bmatrix}
\end{gather*}

In terms of the rows above the main diagonal, an element $m_{i,j}$ of this matrix will be $0$ when $j-i$ is odd and $j \geq i$ where $1 \leq i, j, \leq 5$. For the columns, we may only be interested in basis polynomials that all have even degree, or odd degree. In this example, Chebyshev polynomials of the first kind of even degree span the  same subspace as $\{1, x^2, x^4\}$ and those of odd degree span the same subspace as $\{x, x^3\}$.

If the change of basis just involves Chebyshev polynomials of the first kind of even degree, the following compressed matrix can suffice.
\begin{gather*}
\begin{bmatrix}
1 &  -1&  \phantom{-}1\\
0 &  \phantom{-}2 & -8\\
0&\phantom{-}0&\phantom{-}8
\end{bmatrix}
\end{gather*}
and similarly for those of odd degree:
\begin{gather*}
\begin{bmatrix}
 1  & -3\\
0 &  \phantom{-}4\\
\end{bmatrix}
\end{gather*}

Another compression of a change of basis matrix expressed using the monomials occurs when vectors of degree less than a fixed number $k$ are omitted. For example, with $M_{MT}$ we may want to exclude the vectors $1$ and $x$ from the basis $\{1, x, x^2, x^3 x^4\}$, i.e. $k = 1$. We can use the following change of basis matrix instead of $M_{MT}$.
\begin{gather*}
\begin{bmatrix}
2 & 0& -8\\
0 & 4 & \phantom{-}0\\
0&0&\phantom{-}8
\end{bmatrix}
\end{gather*}

Combinations of these can be used too, e.g. Chebyshev polynomials of even degree excluding terms of degree less than $2$, which has basis in the monomials of $\{x^2, x^4\}$, and change of basis matrix
\begin{gather*}
\begin{bmatrix}
2 & -8\\
0&\phantom{-}8
\end{bmatrix}
\end{gather*}

Equation (\ref{alpha-new2}) is an optimization for finding the last column of a $m \times m$ matrix $M_{tv} = M_{ts}M_{sv}$ where $s, t$ and $v$ are triangular bases of the same vector space. The elements below the main diagonal are excluded. We now consider four specific further optimizations of equation (\ref{alpha-new2}) when other elements of the matrices can be excluded. These cases are:
\begin{enumerate}
\item Exclude the first $k_1$ rows and first $k_1$ columns
\item Exclude the last $k_2$ rows and last $k_2$ columns
\item Exclude  elements $m_{i, j}$ where $j - i$ is odd, $1 \leq i, j \leq m$ and $j \geq i$
\item Exclude elements $m_{i, j}$ where $1 \leq i, j \leq m$,  $m$ is odd and either $i$ or $j$ is even.
\end{enumerate}

An upper triangular $m \times m$ matrix can be represented with $\frac12 (m^2 + m)$ elements. The four  optimizations above 
further reduce the number of elements required. The number of elements that are included are, respectively:

\begin{align}
&\:\:\:\:\:\:\:\frac12 ((m - k_1)^2 + (m - k_1)) \label{opt1}\\
&\:\:\:\:\:\:\:\frac12 ((m - k_2)^2 + (m - k_2))\label{opt2}\\
& \left\{
\begin{array}{ll} 
\frac14 m(m+2) & \mbox{if $m$ is even} \label{opt3}\\
\\
\frac14 (m+1)^2 & \mbox{if  if $m$ is odd}\\
\end{array}
\right. \\
  &\left\{
\begin{array}{ll} 
\frac18 m (m+2) & \mbox{if $m$ is even} \label{opt4}\\
\\
\frac18 (m+1)(m+3) & \mbox{if  if $m$ is odd}\\
\end{array}
\right.
\end{align}

The last two expressions for the included elements can be shown by mathematical induction.

In equations (\ref{c12}) to (\ref{c4}) below, the coefficient functions \ are with respect to the original $m \times m$ matrices. 
In terms of bases, examples of these optimizations occur when changing bases between classical orthogonal polynomials. Case 1 involves omitting the terms of $t$ and $v$ of degree not exceeding $k_1 -1$. Case 2 occurs when omitting terms of $t$ and $v$ of degree at least $m - k_2$ When at least one of $t$ and $v$ has polynomials of definite parity and includes all polynomials of degree less than $m$ then case 3 applies. Case 4 occurs when both $t$ and $v$ have definite parity and we are interested only in polynomials either of even or odd degree.

For cases 1 and 2 we have
\begin{equation} \label{c12}
\alpha_3(n, k)= \sum_{v=0}^{k} \alpha_1(n  - v, k-v) \alpha_2(n , v)
\end{equation} where $n = m - k_2 - 1$ and $0 \leq k  \leq n - k_1$.
\medskip

In case 3, the dimensions of the matrices are not decreased. We have
\begin{equation} \label{c3}
\alpha_3(n, k)= \sum^{k}_{\mathclap{\substack{v = 0 \\ k-v \mbox{\small \; even}}}} \alpha_1(n  - v, k-v) \alpha_2(n , v)
\end{equation} where $n = m  - 1$ and $0 \leq k  \leq n$.

In case 4,  we can assume that $m$ is odd without loss of generality. Otherwise, if $m$ is even, the first column would be excluded and the first row would alternate between excluded elements and $0$. It follows that the first row can also be excluded, which is case 1, above with $k_1 = 1$.

The dimensions of $M_{ts}$ after elements are excluded is $\lceil \frac{m}2\rceil$.


\begin{equation}  \label{c4}
\alpha_3(n, k)= \sum_{v=0}^{k} \alpha_1(n -2v, 2(k-v)) \alpha_2(n , 2v)
\end{equation} where $n = m-1$ and $0 \leq k \leq \frac{n}2$.

We can express equation (\ref{c4}) using coefficient functions for the matrices with excluded elements.

\begin{definition}
\begin{equation} \label{alpha-dash}
\beta(n, k) = \alpha(n, 2k) 
\end{equation} where $0 \leq k \leq \lfloor \frac{n}2 \rfloor$.
\end{definition}

We then obtain equation (\ref{alpha-new2}) with respect to the matrices with excluded elements:
\[
\beta_3(n, k)= \sum_{v=0}^{k} \beta_1(n - 2v, k-v) \beta_2(n , v)
\] where $0 \leq k \leq \frac{n}2$.

\medskip

For example, given 
\begin{gather*}
\begin{bmatrix}
\vdots & \cdots & \vdots & \vdots &\vdots & \vdots & \vdots  \\
0 & \cdots & \alpha_1(n-4, 0) & \alpha_1(n-3,1) & \alpha_1(n-2, 2) & \alpha_1(n-1, 3)& \alpha_1(n ,4)\\
0 & \cdots & 0& \alpha_1(n-3, 0) & \alpha_1(n-2, 1) & \alpha_1(n-1, 2)& \alpha_1(n ,3)\\
0 & \cdots &0 &0&  \alpha_1(n-2, 0) & \alpha_1(n-1, 1)& \alpha_1(n ,2)\\
0 & \cdots &0 & 0&  0 & \alpha_1(n-1, 0)& \alpha_1(n ,1)\\
0 & \cdots & 0 & 0 & 0 & 0 & \alpha_1(n, 0)
\end{bmatrix}
\begin{bmatrix}
\vdots \\
\alpha_2(n, 4) \\
\alpha_2(n, 3) \\
\alpha_2(n, 2)\\
\alpha_2(n, 1) \\
\alpha_2(n, 0)
\end{bmatrix}
\end{gather*}

and $n = 8$, suppose $k=2$, so that
\begin{align*}
\alpha_3(8, 4)=& \alpha_1(8, 4) \alpha_2(8, 0) + \alpha_1(6, 2) \alpha_2(8, 2) + \alpha_1(4, 0) \alpha_2(8, 4) \\
\end{align*} with respect to the matrices with excluded elements

\begin{gather*}
\begin{bmatrix}
 \vdots & \cdots & \vdots & \vdots & \vdots \\
 0 & \cdots &\alpha_1(n-4, 0) & \alpha_1(n-2, 2)& \alpha_1(n ,4)\\
 0 & \cdots &0 & \alpha_1(n-2, 0)& \alpha_1(n ,2)\\
 0 & \cdots &0 & 0 & \alpha_1(n, 0)
\end{bmatrix}
\begin{bmatrix}
\vdots \\
\alpha_2(n, 4) \\
\alpha_2(n, 2)\\
\alpha_2(n, 0)
\end{bmatrix}
\end{gather*}

This gives
\[
\beta_3(8, 2)  = \beta_1(8, 2) \beta_2(8, 0) + \beta_1(6, 1) \beta_2(8, 1) + \beta_1(4, 0) \beta_2(8, 2),
\] which is the dot product with respect to the coefficient functions $\beta_1(x,v)$ and  $\beta_2(x,v)$ by applying equation
(\ref{alpha-dash}):

\begin{gather*}
\begin{bmatrix}
 \vdots & \cdots & \vdots & \vdots & \vdots \\
 0 & \cdots & \beta_1(n-4, 0) & \beta_1(n-2, 1)& \beta_1(n ,2)\\
 0 & \cdots & 0 & \beta_1(n-2, 0)& \beta_1(n ,1)\\
 0 & \cdots & 0 & 0 & \beta_1(n, 0)
\end{bmatrix}
\begin{bmatrix}
\vdots \\
\beta_2(n, 2) \\
\beta_2(n, 1)\\
\beta_2(n, 0)
\end{bmatrix}
\end{gather*}

\section{Examples with Orthogonal Polynomials}

Definitions of the the classical orthogonal polynomials are given in  Koornwinder et al.~\cite[\S 18.3]{dlmf}.  Three of these kinds of orthogonal polynomials are Chebyshev Polynomials of the first kind, $T_n(x)$, Laguerre polynomials $L_n(x)$, and Physicist's Hermite polynomials $H_n(x)$. There do not appear to be results on changing basis from Laguerre or Hermite polynomials to Chebyshev polynomials of the first kind.

The polynomials $T_n(x)$ and $H_n(x)$ have definite parity~\cite[Definition 21, \S 8.2]{boyd}, and the Laguerre polynomials $L_n(x)$ do not have definite parity.   $T_n(x)$ $L_n(x)$ and $H_n(x)$ all have degree $n$ where $n \geq 0$.

\subsection{Example of Case 3}

If we want to change basis from the Laguerre polynomials up to degree $5$ to Chebyshev polynomials of the first kind, we can use the change of basis equation $M_{TL} = M_{TM} M_{ML}$ where
\begin{center}
$M_{TM} = $
\begin{doublespace}
\noindent\(\left[
\begin{array}{cccccccc}
 1 & 0 & \frac{1}{2} & 0 & \frac{3}{8} & 0  \\
 0 & 1 & 0 & \frac{3}{4} & 0 & \frac{5}{8}  \\
 0 & 0 & \frac{1}{2} & 0 & \frac{1}{2} & 0  \\
 0 & 0 & 0 & \frac{1}{4} & 0 & \frac{5}{16}  \\
 0 & 0 & 0 & 0 & \frac{1}{8} & 0  \\
 0 & 0 & 0 & 0 & 0 & \frac{1}{16}  \\
\end{array}
\right]\)
\medskip

and $M_{ML} =$
\noindent\(\left[
\begin{array}{cccccccc}
 1 & \phantom{-}1 & \phantom{-}1 & \phantom{-}1 &  \phantom{-}1 &  \phantom{-}1  \\
 0 & -1 & -2 & -3&  -4 &  -5  \\
 0 & \phantom{-}0 &  \phantom{-}\frac12 &  \phantom{-}\frac32 &  \phantom{-}3 &  \phantom{-}5  \\
 0 & \phantom{-}0 &  \phantom{-}0 &  -\frac16 &  -\frac{2}3 & -\frac53  \\
 0 & \phantom{-}0 & \phantom{-}0 &  \phantom{-}0 &  \phantom{-}\frac{1}{24} &  \phantom{-}\frac{5}{24}  \\
 0 & \phantom{-}0 &  \phantom{-}0 &  \phantom{-}0 &  \phantom{-}0 &  -\frac{1}{120}  \\
\end{array}
\right]\)
\end{doublespace}
\end{center}

The product is
\begin{center}
$M_{TL} = $
\begin{doublespace}
\noindent\(\left[
\begin{array}{cccccccc}
 1 & \phantom{-}1 & \phantom{-}\frac54 & \phantom{-}\frac74 &  \phantom{-}\frac{161}{64} &  \phantom{-}\frac{229}{64} \\
 0 & -1 & -2 & -\frac{25}8&  -\frac92 &  -\frac{1201}{192}  \\
 0 & \phantom{-}0 &  \phantom{-}\frac14 &  \phantom{-}\frac34 &  \phantom{-}\frac{73}{48} &  \phantom{-}\frac{125}{48}  \\
 0 & \phantom{-}0 &  \phantom{-}0 &  -\frac1{24} &  -\frac{1}6 & -\frac{161}{384}  \\
 0 & \phantom{-}0 & \phantom{-}0 &  \phantom{-}0 &  \phantom{-}\frac{1}{192} &  \phantom{-}\frac{5}{192}  \\
 0 & \phantom{-}0 &  \phantom{-}0 &  \phantom{-}0 &  \phantom{-}0 &  -\frac{1}{1920}  \\
\end{array}
\right]\)
\end{doublespace}
\end{center} and, for example, from the fifth column, we have
\begin{align*}
L_4(x) =& \frac{1}{192}T_4(x) - \frac16T_3(x) + \frac{73}{48} T_2(x) - \frac92 T_1(x) + \frac{161}{64} T_0(x)\\
=&\frac{x^4}{24} - \frac{2 x^3}3 + 3 x^2 - 4x + 1
\end{align*}

Since $L_n(x)$ is not definite and $T_n(x)$ we can only exclude the elements that are zero in the upper triangle of $M_{TM}$ in finding $\alpha_3(n,k)$. The dimensions of the matrices are unchanged.

From Mason and Handscomb~\cite[equation (2.14)]{handscomb}, and equivalently Tao~\cite{tao}
\begin{equation}
x^n = 2^{1-n} {\sum_{k=0}^{\lfloor \frac{n}2 \rfloor}}{}^{'} {{n} \choose {k}} T_{n-2k}(x)
\end{equation} where the primed sum means  that its $k$th term is halved if $n$ is even and $k = \frac{n}2$.

We have
\[
\alpha_1(n,k) = 
\left\{
\begin{array}{ll} 
2^{1-n} {n \choose {k/2}} & \mbox{if $k$ is even and $k \not=n$}\\
2^{-n} {n \choose {k/2}} & \mbox{if $k$ is even and $k = n$}\\
0 & \mbox{if $k$ is odd.}
\end{array}
\right.
\] where $0 \leq k \leq n$.

From equation 18.5.12 of Koornwinder et al.~\cite{dlmf} for the generalized Laguerre polynomials
\[
L_n^{(\alpha)}(x) = \sum_{l=0}^n \frac{(\alpha + l+1)_{n-l}}{(n-l)! \; l!}(-x)^l
\] so that a change of basis matrix from the Laguerre polynomials to the monomials has the coefficient function
\begin{equation*} \label{cfL}
\alpha_2(n, k) = (-1)^{n-k} \frac{(n - k + 1)_{k}}{k! \; (n-k)!}
\end{equation*} where $(n - k + 1)_{k}$ is an application of the Pochhammer symbol or rising factorial.

Applying case 3, gives from equation (\ref{c3}),
\[
\alpha_3(n,k) = \sum^{k}_{\mathclap{\substack{v = 0 \\ k-v \mbox{\small \; even}}}} \alpha_1(n  - v, k-v) \alpha_2(n , v)
\] where $0 \leq k \leq n$.

We have
\[
L_n(x) = \sum_{v=0}^n \alpha_3(n, v) T_{n-v}(x)
\] where $0 \leq k \leq n$. From equation~(\ref{opt3}), the number of included elements is $\frac14 (6 \times 8) = 12$ which can also be seen from $M_{TM}$.

\subsection{Example of Cases 1 and 4}
Suppose that we want to change basis from Physicist's Hermite polynomials of  degree up to $7$ to Chebyshev polynomials of the first kind. A change of basis equation is $M_{TH} = M_{TM} M_{MH}$ where 

\begin{center}
$M_{TM} = $
\begin{doublespace}
\noindent\(\left[
\begin{array}{cccccccc}
 1 & 0 & \frac{1}{2} & 0 & \frac{3}{8} & 0 & \frac{5}{16} & 0 \\
 0 & 1 & 0 & \frac{3}{4} & 0 & \frac{5}{8} & 0 & \frac{35}{64} \\
 0 & 0 & \frac{1}{2} & 0 & \frac{1}{2} & 0 & \frac{15}{32} & 0 \\
 0 & 0 & 0 & \frac{1}{4} & 0 & \frac{5}{16} & 0 & \frac{21}{64} \\
 0 & 0 & 0 & 0 & \frac{1}{8} & 0 & \frac{3}{16} & 0 \\
 0 & 0 & 0 & 0 & 0 & \frac{1}{16} & 0 & \frac{7}{64} \\
 0 & 0 & 0 & 0 & 0 & 0 & \frac{1}{32} & 0 \\
 0 & 0 & 0 & 0 & 0 & 0 & 0 & \frac{1}{64} \\
\end{array}
\right]\)
\medskip

and $M_{MH} =$
\noindent\(\left[
\begin{array}{cccccccc}
 1 & 0 & -2 &  0 &  12 &  0 & -120 &  0 \\
 0 & 2 & 0 & -12 &  0 &  120 & 0 & -1680 \\
 0 & 0 &  4 &  0 & -48 &  0 &  720 &  0 \\
 0 & 0 &  0 &  8 &  0 & -160 &  0 &  3360 \\
 0 & 0 &  0 &  0 &  16 &  0 & -480 &  0 \\
 0 & 0 &  0 &  0 &  0 &  32 &  0 & -1344 \\
 0 & 0 &  0 &  0 &  0 &  0 &  64 &  0 \\
 0 & 0 &  0 &  0 &  0 &  0 &  0 &  128 \\
\end{array}
\right]\)
\end{doublespace}
\end{center}

If we want to  change basis for polynomials of odd degree, this is an example of cases 1 and 4 in \S \ref{opt-subspace}. The first column of $M_{TM}$ can be excluded because it corresponds to the representation of monomial $1$ which has even degree. The first row can also be excluded because the non-zero elements are coefficients of $T_0(x)$ in the respective representations of $\{x, x^2, \ldots, x^7\}$, and $T_0(x)$ has even parity. 

Excluding the first rows and columns gives the following $7 \times 7$ matrices, respectively.
\begin{center}
\begin{doublespace}

\noindent\(\left[
\begin{array}{cccccccc}
  1 & 0 & \frac{3}{4} & 0 & \frac{5}{8} & 0 & \frac{35}{64} \\
 0 & \frac{1}{2} & 0 & \frac{1}{2} & 0 & \frac{15}{32} & 0 \\
 0 & 0 & \frac{1}{4} & 0 & \frac{5}{16} & 0 & \frac{21}{64} \\
 0 & 0 & 0 & \frac{1}{8} & 0 & \frac{3}{16} & 0 \\
 0 & 0 & 0 & 0 & \frac{1}{16} & 0 & \frac{7}{64} \\
 0 & 0 & 0 & 0 & 0 & \frac{1}{32} & 0 \\
 0 & 0 & 0 & 0 & 0 & 0 & \frac{1}{64} \\
\end{array}
\right]\)
\medskip

\noindent\(\left[
\begin{array}{cccccccc}
 2 & 0 & -12 &  0 &  120 & 0 & -1680 \\
 0 &  4 &  0 & -48 &  0 &  720 &  0 \\
  0 &  0 &  8 &  0 & -160 &  0 &  3360 \\
 0 &  0 &  0 &  16 &  0 & -480 &  0 \\
 0 &  0 &  0 &  0 &  32 &  0 & -1344 \\
0 &  0 &  0 &  0 &  0 &  64 &  0 \\
0 &  0 &  0 &  0 &  0 &  0 &  128 \\
\end{array}
\right]\)
\end{doublespace}
\end{center}

From equation~(\ref{opt1}), this results in $\frac12 (7^2 + 7) = 28$ elements in the  upper triangles of these matrices.
By applying equation~(\ref{c12}), we have $k_1 = 1$, $k_2 = 0$, $m = 8$, $n = 7$ and $0 \leq k \leq 6$, so that
\[
\alpha_3(7, k) = \sum_{v=0}^{k} \alpha_1(7  - v, k-v) \alpha_2(7 , v).
\]

We have, for example,
\begin{align*}
\alpha_3(7, 3) =& \alpha_1(7, 3) \alpha_2(7, 0) + \alpha_1(6, 2) \alpha_2(7, 1) + \alpha_1(5, 1) \alpha_2(7, 2) + \alpha_1(4, 0) \alpha_2(7, 3)\\
=&  0 \times \alpha_2(7, 0) + \alpha_1(6, 2) \times 0  + 0 \times \alpha_2(7, 2) + \alpha_1(4, 0) \times 0\\
= & 0
\end{align*}
and
\begin{align*}
\alpha_3(7, 2) =& \alpha_1(7, 2) \alpha_2(7, 0) + \alpha_1(6, 1) \alpha_2(7, 1) + \alpha_1(5, 0) \alpha_2(7, 2)\\
=&  \frac{7}{64}  \times 128  + 0 \times 0  + \frac{1}{16}  \times -1344\\
= & -70
\end{align*}

The elements of the matrices that are zero are known from the outset.

The result of applying case 4 is to exclude these elements.
In the resulting $7 \times 7$ matrix, alternate rows and columns can be excluded: rows 2, 4, and 6, and columns 2, 4 and 6 because they correspond to the monomials $x^2$, $x^4$ and $x^6$ and their representations in the basis $\{T_1(x) , \ldots, T_7(x)\}$

The matrix $M_{MH}$ has the same pattern of zero and non-zero elements as $M_{TM}$, so that  elements that are zero in the upper triangle can be excluded from both matrices.

After applying case 1 and 4, we obtain the compressed  matrices which  have $\frac{1}4$ of the total elements, and $\frac5{32}$ of  them are included.
\begin{gather*}
\begin{bmatrix}
1 & \frac34 & \frac58 & \frac{35}{64}\\
\\
0 & \frac14 & \frac5{16} & \frac{21}{64}\\
\\
0 & 0 & \frac{1}{16} & \frac{7}{64}\\
\\
0&0&0& \frac1{64}
\end{bmatrix}
\begin{bmatrix}
2 & -12 & \phantom{-}120 & -1680\\
\\
0 & \phantom{-}8 & -160 & \phantom{-}3360\\
\\
0 & \phantom{-}0 & \phantom{-}32 & -1344\\
\\
0 & \phantom{-}0 & \phantom{-}0 & 128
\end{bmatrix}
=
\begin{bmatrix}
2 & -6 & \phantom{-}20 & \phantom{-}70\\
\\
0 & \phantom{-}2 & -30 &  \phantom{-}462\\
\\
0 & \phantom{-}0 & \phantom{-}2 & -70\\
\\
0&  \phantom{-}0 &  \phantom{-}0 &  \phantom{-}2
\end{bmatrix}
\end{gather*}

From equation~(\ref{opt4}), the number of included elements in each matrix after the optimization is $\frac18 (8 \times 10) = 10$.
From the third column of the last matrix, for example,
\begin{align*}
H_5(x) =& 2 T_5(x) - 30 T_3(x) + 20 T_1(x), \mbox{ i.e.}\\
32 x^5 - 160 x^3 + 120 x =& 2 (16x^5 - 20 x^3 + 5x) - 30 (4x^3 - 3x) + 20 x.
\end{align*}

Similarly, if we  we want to  change basis for polynomials of odd degree, this is an example of cases 2 and 4 in \S \ref{opt-subspace} where the last column and row of the matrices $M_{TM}$ and $M_{MH}$ are excluded, and also elements that are zero in the upper triangles. This gives

\begin{gather*}
\begin{bmatrix}
1 & \frac12 & \frac38 & \frac{5}{16}\\
\\
0 & \frac12 & \frac12 & \frac{15}{32}\\
\\
0 & 0 & \frac{1}{8} & \frac{3}{16}\\
\\
0&0&0& \frac1{32}
\end{bmatrix}
\begin{bmatrix}
1 & -2 & \phantom{-}12 & -120\\
\\
0 & \phantom{-}4 & -48 & \phantom{-}720\\
\\
0 & \phantom{-}0 & \phantom{-}16 & -480\\
\\
0 & \phantom{-}0 & \phantom{-}0 &  \phantom{-}64
\end{bmatrix}
=
\begin{bmatrix}
1 & 0 & -6 & \phantom{-}80\\
\\
0 & \phantom{-}2 & -16 &  \phantom{-}150\\
\\
0 & \phantom{-}0 & \phantom{-}2 & -48\\
\\
0&  \phantom{-}0 &  \phantom{-}0 &  \phantom{-}2
\end{bmatrix}
\end{gather*}

From the fourth column of the last matrix, for example, we have
\begin{align*}
H_6(x) =& 2 T_6(x) - 48 T_4(x) + 150 T_2(x) + 80 T_0(x)\\
=& 64 x^6 -480 x^4 + 720 x^2 - 120.
\end{align*}

 because the columns corresponding to monomials of even degree can be excluded, i.e. columns 2, 4, and 6
The vector space we use has the following basis represented by the monomials: $B = \{x, x^3, x^5, x^7\}$. The bases $\{H_1(x), H_3(x), H_5(x), H_7(x)\}$ and $\{T_1(x), T_3(x), T_5(x), T_7(x)\}$ are both triangular bases with respect to $B$.

From~\cite[equation 18.5.12]{dlmf}, we have
\begin{equation} \label{Hfun}
H_n(x) = n! \sum_{l=0}^{\lfloor \frac{n}2 \rfloor} \frac{{(-1)}^l {(2x)}^{n-2l}}{l! (n- 2l)!}
\end{equation}
For this example, it gives
\[
\alpha_2(n ,k) = 
\left\{
\begin{array}{ll} 
n! \frac{(-1)^{\frac{k}2} 2^{n - k}}{\frac{k}{2}! (n- k)!} & \mbox{if $k$ is even}\\
0 & \mbox{if $k$ is odd.}
\end{array}
\right.
\] where $0 \leq k \leq n$ because $H_n(x)$ has definite parity.

From equation~(\ref{alpha-dash}) we obtain,
\[
\beta_2(n ,k) = 
n! \frac{(-1)^{k} 2^{n - 2k}}{k! (n- 2k)!} 
\] where $0 \leq k \leq \lfloor \frac{n}2 \rfloor$.

This gives the following change of basis matrix $M_{B H}$:
\begin{gather*}
\begin{bmatrix}
2 & -12 & \phantom{-}120 & -1680\\
\\
0 & \phantom{-}8 & -160 & \phantom{-}3360\\
\\
0 & \phantom{-}0 & \phantom{-}32 & -1344\\
\\
0 & \phantom{-}0 & \phantom{-}0 & \phantom{-}128
\end{bmatrix}
\end{gather*}

From Mason and Handscomb~\cite[equation (2.14)]{handscomb}, and equivalently Tao~\cite{tao}
\begin{equation}
x^n = 2^{1-n} {\sum_{k=0}^{\lfloor \frac{n}2 \rfloor}}{}^{'} {{n} \choose {k}} T_{n-2k}(x)
\end{equation} where the primed sum means  that its $k$th term is halved if $n$ is even and $k = \frac{n}2$.

We have
\[
\alpha_1(n,k) = 
\left\{
\begin{array}{ll} 
2^{1-n} {n \choose {\frac{k}2}} & \mbox{if $k$ is even}\\
0 & \mbox{if $k$ is odd.}
\end{array}
\right.
\] where $0 \leq k \leq n$ and $n$ is odd.

From equation (\ref{alpha-dash}), it follows that
\[
\beta_1(n, k) = 2^{1-n} {n \choose k}
\]
where $0 \leq k \leq \lfloor \frac{n}2 \rfloor$.

This gives change of basis matrix $M_{T B}$:
\begin{gather*}
\begin{bmatrix}
1 & \frac34 & \frac58 & \frac{35}{64}\\
\\
0 & \frac14 & \frac5{16} & \frac{21}{64}\\
\\
0 & 0 & \frac{1}{16} & \frac{7}{64}\\
\\
0&0&0& \frac1{64}
\end{bmatrix}
\end{gather*}

We have $M_{TH} = M_{TB} M_{BH}$:
\begin{gather*}
\begin{bmatrix}
2 & -6 & \phantom{-}20 & \phantom{-}70\\
\\
0 & \phantom{-}2 & -30 &  \phantom{-}462\\
\\
0 & \phantom{-}0 & \phantom{-}2 & -70\\
\\
0&  \phantom{-}0 &  \phantom{-}0 &  \phantom{-}2
\end{bmatrix}
= 
\begin{bmatrix}
1 & \frac34 & \frac58 & \frac{35}{64}\\
\\
0 & \frac14 & \frac5{16} & \frac{21}{64}\\
\\
0 & 0 & \frac{1}{16} & \frac{7}{64}\\
\\
0&0&0& \frac1{64}
\end{bmatrix}
\begin{bmatrix}
2 & -12 & \phantom{-}120 & -1680\\
\\
0 & \phantom{-}8 & -160 & \phantom{-}3360\\
\\
0 & \phantom{-}0 & \phantom{-}32 & -1344\\
\\
0 & \phantom{-}0 & \phantom{-}0 & 128
\end{bmatrix}
\end{gather*}

From the third column,
\begin{align*}
H_5(x) =& 2 T_5(x) - 30 T_3(x) + 20 T_1(x), \mbox{ i.e.}\\
32 x^5 - 160 x^3 + 120 x =& 2 (16x^5 - 20 x^3 + 5x) - 30 (4x^3 - 3x) + 20 x.
\end{align*}

The columns of $M_{TH}$ for any $n> 0$ where $n$ is odd can be found by applying equation (\ref{alpha-new2}). The maximum degree of the basis polynomials is $m = 2n - 1$.
\begin{align*} 
\beta_3(n, k)=& \sum_{v=0}^{k} \beta_1(n - 2v, k-v) \beta_2(n , v)\\
=& \sum_{v=0}^{k} 2^{1-n+2v} {{n-2v} \choose {k-v}} \frac{n! {(-1)}^v 2^{n -2v}}{v! (n-2v)!}\\
=& 2 \sum_{v=0}^{k} {(-1)}^v   \frac{n!}{v! \; (k-v)! \; (n - k-v)!}     
\end{align*} where $0 \leq k \leq \lfloor \frac{n}2 \rfloor$.

This gives
\begin{equation}
H_n(x) = \sum_{k=0}^{\lfloor \frac{n}2 \rfloor} \beta_3(n, k) T_{n-2k}(x)
\end{equation}  where $n> 0$ and $n$ is odd.

\section{Conclusion}
Change of basis in finite vector spaces has numerous significant and widespread applications.   We  show in Theorem~\ref{itsagroupoid} that in general, the change of basis matrices of a set of $m$ bases of a finite vector space is a connected groupoid of order $m^2$. The groupoid function is matrix multiplication between matrices that have an exchange basis, following Definition~\ref{exbasis}. These change of basis matrices represent transformations between bases and can be typed with simple types.

We are concerned with finding  algebraic expressions for the elements of such change of basis matrices.
The groupoid function leads to a general method via a vector dot product for this. Five optimizations of it are defined.

Firstly, Definition~\ref{triangular-basis} of triangular basis specializes change of basis matrices to upper triangular ones. 
The algebraic expressions in change of basis matrices can be expressed by coefficient functions, $\alpha(n, k)$, following Definition~\ref{mapF}. Equation (\ref{alpha-new2}) provides the coefficient functions in this case.   This optimization has similarities  to  that by Schuster~\cite{schuster} and Dilts~\cite{dilts} in their algorithms that use spherical harmonics~\cite{boyd}.

Four additional optimizations of the evaluation of $\alpha(n, k)$ were defined in Section~\ref{opt-subspace}. They exclude elements that are known to be $0$ in the evaluation of vector dot products, or exclude rows and columns from a change of basis matrix.  Expressions are given for the number of included elements of a change of basis matrix.

These optimizations are applicable, for example, in the context of bases that are classical orthogonal polynomials and when the parity of basis polynomials is relevant, such as approximating even or odd functions. Examples of the optimizations are given with a change of basis from Hermite polynomials $H_n(x)$ of odd parity to Chebyshev polynomials $T_n(x)$. Another example is a change of basis from Laguerre polynomials  $L_n(x)$, that do not have definite parity, to $T_n(x)$.

\section*{Acknowledgment}
 I am grateful  to the College of Engineering \& Computer Science at The Australian National University for research support.

\end{document}